\documentclass[10pt]{amsart}
\usepackage{fullpage} 
\usepackage{amssymb} 
\usepackage{graphicx} 
\usepackage{amsmath,amsthm,amssymb,latexsym} 
\usepackage{enumerate} 
\usepackage{tikz}
\usepackage{accents} 
\usepackage[bookmarks,colorlinks,breaklinks]{hyperref}
\hypersetup{linkcolor=black,citecolor=black,filecolor=black,urlcolor=black}

\newtheorem{theorem}{Theorem}[section]

\theoremstyle{definition}

\theoremstyle{remark}

\numberwithin{equation}{section}

\begin{document}

\setcounter{page}{1}
\title[]{A necessary and sufficient criterion for the existence of ratio limits of sequences generated by linear recurrences}
\author{Igor Szczyrba}
\address{School of Mathematical Sciences\\
                University \!of Northern Colorado\\
                Greeley CO 80639, U.S.A.}
\email{igor.szczyrba@unco.edu}

\begin{abstract}
We introduce a necessary and sufficient criterion for determining the existence and the values of ratio limits of complex sequences generated by arbitrary linear recurrences. 
\end{abstract}
 
\maketitle

\section{Introduction}

Sloane's Online Encyclopedia of Integer Sequences \cite{Sloane} and Khovanova's website \cite{Khovanova} catalog thousands of integer sequences generated by linear recurrences that are associated with problems in various branches of mathematics and other sciences, such as number theory, abstract algebra, linear algebra, combinatorics, complex numbers, group theory, probability, statistics, affine geometry, electrical networks, infectious diseases, etc.,  cf.\  \cite{Azarian,  Balof, Baltic, Cahill, Chaffin, Dubeau1, Dubeau4, Dubeau3, Falcon,  Guo, Huang, Janjic, Linton, Noe, Rajan, Shattuck, Szczyrba1, Szczyrba2, Szczyrba3}. 

The asymptotic behavior of sequences generated by linear recurrences is characterized by the ratio limit of the sequence's consecutive terms. The knowledge on whether a ratio limit exists is necessary if a problem  requires considering the sequence's terms with higher and higher indices.  
The existence of a ratio limit and its value depend on the choice of the sequence's initial conditions. 

In 1997 Dubeau et al.\ \cite{Dubeau} studied linear recurrences $F\in L(\mathbb{C}^n,\mathbb{C}^n)$ with asymptotically simple characteristic polynomials 
\begin{equation}
P=\lambda^{n}-b_1\lambda^{n-1}-\cdots-b_{n}, \quad b_n\!\neq\!0.
\end{equation}\label{10}
A polynomial is asymptotically simple iff among its zeros of maximal modulus there is a dominant zero $\lambda_0$ of maximal multiplicity.
  
Dubeau et al.\ derived {\em a sufficient\,} criterion for the existence of ratio limits of sequences $\big(F^{\bf a}_k\big)_{k=-n+1}^\infty$ generated by $F$ from complex initial conditions ${\bf a}=(a_{-n+1},\dots,a_0)$. 
Specifically, the authors showed that if  

\begin{equation}
a_0\lambda_0^{n-1}+\sum_{i=1}^{n-1}a_{-i}\sum^{n-i}_{j=1}b_{i+j}\lambda_0^{n-j-1}\ne0,\label{14} 
\end{equation}

\begin{equation}
\text{then}\quad\lim_{k_0<k\to\infty}\frac{F^{\bf a}_{k+1}}{F^{\bf a}_k}=\lambda_0,\,\,\,F^{\bf a}_k\ne0\,\,\,\text{for}\,\,\,k>k_0,\label{12}
\end{equation} 

\begin{equation}
\text{where}\,\,\,  
  F^{\bf a}_k=a_k\,\,\,\text{if}\,\,-n+1\le k\le0\,\,\,\text{and}\,\,\, F^{\bf a}_k = b_1F^{\bf a}_{k-1}+\cdots + b_nF^{\bf a}_{k-n}\,\,\,\text{if}\,\,k>0. \label{11}
\end{equation}
Condition \eqref{14} is satisfied, in particular, by all sequences generated by linear recurrences with asymptotically simple characteristic polynomials from initial conditions $(0,\dots,0,a_0)$. 

An example of a sequence that does not satisfy condition \eqref{14}, but has the ratio limit, is the constant sequence $(1_k)_{k=-n+1}^\infty$ generated by
the linear recurrence with the signature $(2,-1)$ from the initial conditions $(1, 1)$. The corresponding asymptotically simple characteristic polynomial $P=(\lambda-1)^2$\!.

We generalize results obtained by Dubeau et al.\ by introducing {\em a necessary and sufficient\,} criterion for the existence of ratio limits of complex sequences generated by linear recurrences with {\em arbitrary\,} characteristic polynomials $P$. We also prove that if the ratio limit exists, it must be equal to one of the zeros of $P$. 

\section{Main results}

Given a linear recurrence $F\in L(\mathbb{C}^n,\mathbb{C}^n)$ of an order $n$ with the signature $(b_1,\dots,b_n)$, where $b_n\!\neq\!0$. A sequence $(F^{\bf a}_k)_{k=-n+1}^\infty=F^{\bf a}$ generated  by formulas \eqref{11} is called {\em a solution of $F$}. 
 
\smallskip \begin{theorem}\label{A}
 If a solution $F^{\bf a}$ of $F$ generated from initial conditions ${\bf a}\in\mathbb{C}^n$ has a ratio limit 

\begin{equation}
\lim_{k_0<k\to\infty} \frac{F^{\bf a}_{k+1}}{F^{\bf a}_k}=\Psi,\quad{\rm where}\quad F^{\bf a}_k\ne0\quad{\rm for}\quad  k>k_0,\label{22}
\end{equation}
then $\Psi$ is equal to one of the zeros of the characteristic polynomial $P$ of $F$.

\end{theorem}

\begin{proof}

If $n=1$, $b_1$ is the zero of the characteristic monomial, and we have $(F^{\bf a}_k)_{k=0}^\infty=a_1\cdot b_1^k$. 

\smallskip If $n>1$, we introduce a continuous mapping $f:(\mathbb C^*)^n\to(\mathbb C^*)^n$ defined as

\begin{equation}
f(z_1,\dots,z_n)=\biggr(z_2,\dots,z_n, b_1+\frac{b_2}{z_{n}}+\frac{b_3}{z_{n-1}z_{n}}+\dots +\frac{b_n}{ z_2\dots z_{n}}\biggr).\label{23}
\end{equation}

\begin{equation}
f\biggr(\frac{F_{k_0+2}}{F_{k_0+1}},\dots,\frac{F_{k_0+n+1}}{F_{k_0+n}}\biggr)\!=\biggr(\frac{F_{k_0+3}}{F_{k_0+2}},\dots,\frac{F_{k_0+n+1}}{F_{k_0+n}},\frac{b_1F_{k_0+n+1}+b_2F_{k_0+n} +\dots + b_n F_{k_0+2}}{F_{k_0+n+1}}\biggr).\label{24}
\end{equation}

\medskip\noindent It follows from formula \eqref{11} with $k=k_0+n+2$, and equation \eqref{24} that

\begin{equation}
f\biggr(\frac{F_{k_0+2}}{F_{k_0+1}},\dots,\frac{F_{k_0+n+1}}{F_{k_0+n}}\biggr)=\biggr(\frac{F_{k_0+3}}{F_{k_0+2}},\dots,\frac{F_{k_0+n+2}}{F_{k_0+n+1}}\biggr).\label{25}
\end{equation}
 
\bigskip Our assumption \eqref{22} and equation \eqref{25} imply that iterations of $f$ create a sequence convergent to the vector $(\Psi,\dots,\Psi)$. Since $f$ is continuous, it means that the vector $(\Psi,\dots,\Psi)$ is a fixed point of $f$, \cite[p.227]{Maurin}, i.e.,  
\begin{equation}
f(\Psi,\dots,\Psi)=(\Psi,\dots,\Psi).\label{26}
\end{equation}

\medskip\noindent  On the other hand, from the continuity of $f\!$, equation \eqref{24}, and the fact that due to \eqref{22} 

\begin{equation}
 \lim_{k_0\to\infty}\frac{F_{k_0+i+1}}{F_{k_0+n+1}}= \Psi^{-n+i}\!,\quad i=1,\dots,n,\label{27}
\end{equation}
we obtain that  

\begin{equation}
f(\Psi,\dots,\Psi)=\big(\Psi,\dots,\Psi,b_1+b_2\Psi^{-1}+\dots+b_n\Psi^{1-n}).\label{28}
\end{equation}

\bigskip \noindent Equations \eqref{26} and \eqref{28} imply that 
$\,\Psi^{n}-b_1\Psi^{n-1}-\cdots-b_{n}=0$.

\end{proof}

\bigskip Let the characteristic polynomial $P$ of a linear recurrence $F$ have $\nu$ distinct zeros. For simplicity of the notation, we label them as $\lambda_i$, $i=1,\dots,\nu$. Let $\mu_i$ denote the multiplicity of the zero $\lambda_i$, i.e., $\sum_{i=1}^\nu\mu_i=n$. 

\medskip Any solution $F^{\bf a}$ of $F$ is a linear combination of the following $n$ {\em basic solutions} of $F$ \cite{Jeske, Kelly}:
\begin{equation}
(k^j\lambda^k_i)_{k=-n+1}^\infty, \quad i=1,\dots, \nu,\quad j=0,\dots, \mu_i-1,\label{29}
\end{equation}

\bigskip\noindent  So, for $k\geq-n+1$, we have
\begin{equation}
  F_k^{\bf a}=\sum_{i=1}^\nu\sum_{j=0}^{\mu_i-1}c_{ij}^{\bf a}k^j\lambda^k_i.\label{30}
\end{equation}
 
The coefficients $(c^{\bf a}_{1,0},\dots,c^{\bf a}_{\nu,\mu_{\nu}-1})={\bf c^{\bf a}}$ are solutions of the system of linear equations 
\begin{equation}
{\bf c^{\bf a}}=C^{-1}{\bf a}, \label{31}
\end{equation}
where columns of matrix $C$ consists of linearly independent vectors built from the initial conditions of basic solutions \eqref{29}, i.e., 

\[C=\left[
\begin{array}{ccccccc}
1               &\cdots&0                               &\cdots&1                   &\cdots&0                                   \\
\lambda_1^{-1}  &\cdots&(-1)^{\mu_1-1}\lambda_1^{-1}    &\cdots&\lambda_{\nu}^{-1}  &\cdots&(-1)^{\mu_\nu-1}\lambda_{\nu}^{-1}    \\
\lambda_1^{-2}  &\cdots&(-2)^{\mu_1-1}\lambda_1^{-2}    &\cdots&\lambda_{\nu}^{-2}  &\cdots&(-2)^{\mu_\nu-1}\lambda_{\nu}^{-2}    \\
\vdots          &\ddots&\vdots                          &\ddots&\vdots              &\ddots&\vdots                              \\
\lambda_1^{-n+1}&\cdots&(-n+1)^{\mu_1-1}\lambda_1^{-n+1}&\cdots&\lambda_{\nu}^{-n+1}&\cdots&(-n+1)^{\mu_\nu-1}\lambda_{\nu}^{-n+1}\\
\end{array}
\right].\]

 \bigskip We define {\em the characteristic polynomial $P^{\bf a}$ of a solution $F^{\bf a}$} as follows:
\begin{itemize}
  \item $P^{\bf a}$ has as its zeros all those zeros $\lambda_i$ of $P$ for which there exists $j$ such that $c_{ij}^{\bf a}\neq 0$; 
  \item The multiplicity of a zero $\lambda_i$ in $P^{\bf a}$ is equal to the largest index $j$ for which $c_{ij}^{\bf a}\neq 0$. 
\end{itemize}

In what follows, we say that a solution $F^{\bf a}$ of a linear recurrence $F$ {\em is associated\,} with the characteristic polynomial $P^{\bf a}$\!. 

\medskip \begin{theorem}\label{B}

Given a solution $F^{\bf a}$ of a linear recurrence $F$. The ratio limit
\begin{equation}
\lim_{k_0<k\to\infty} \frac{F^{\bf a}_{k+1}}{F^{\bf a}_k},\quad{\rm where}\quad F^{\bf a}_k\ne0\quad{\text for}\quad  k>k_0,\label{33}
\end{equation}
exists iff the characteristic polynomial $P^{\bf a}$ of the solution $F^{\bf a}$ is asymptotically simple.   

If the latter is true, then 
\begin{equation}
\lim_{k_0<k\to\infty} \frac{F^{\bf a}_{k+1}}{F^{\bf a}_k}=\lambda_{i_0}, \label{34} 
\end{equation}
where $\lambda_{i_0}$ is the dominant zero of $P^{\bf a}\!$.\footnote{Condition \eqref{14} ensures that the dominant zero $\lambda_{i_0}$ of $P^{\bf a}\neq P$ coincides with the dominant zero $\lambda_0$ of $P$.}

\end{theorem}

\begin{proof}

$\Longleftarrow$ Let $\lambda_{i_0}$ be the dominant zero with the multiplicity $j_0$ of the asymptotically simple polynomial $P^{\bf a}\!$. Then, we have

\begin{equation}
\lim_{k\to\infty}\frac{F_k^{\bf a}}{k^{j_0}\lambda_{i_0}^k}=c_{{i_0}{j_0}}.\label{35}
\end{equation}
Formula \eqref{35} implies that

\begin{equation}
\lim_{k_0<k\to\infty}\frac{F_{k+1}^{\bf a}}{F_{k}^{\bf a}}=\lambda_{i_0}\cdot\lim_{k_0<k\to\infty}\biggr(\frac{F_{k+1}^{\bf a}}{(k+1)^{j_0}\lambda_{i_0}^{k+1}}\cdot\frac{k^{j_0}\lambda_{i_0}^k}{F_k^{\bf a}}\biggr)=\lambda_{i_0}.\label{36}
\end{equation}

\bigskip $\Longrightarrow$ Let us assume that the ratio limit exists and that the characteristic polynomial $P^{\bf a}$ of a solution $F^{\bf a}$ is not asymptotically simple. Then $P^{\bf a}$ has $\eta>1$ {\em distinct\,} zeros, say $\lambda_1,\dots,\lambda_{\eta}$, such that

\begin{enumerate}[{(i)}]
\item the modulus $R=|\lambda_1|=\dots=|\lambda_\eta|\,$ is greater than or equal to the moduli of other zeros of the polynomial $P^{\bf a}$; and 
\item there exist nonzero coefficients $c^{\bf a}_{1j_0},\dots,c^{\bf a}_{\eta j_0}$ with the index $j_0$ greater than all indices $j$ corresponding to zeros of $P^{\bf a}$ with the same modulus as $R$. 
\end{enumerate}

We decompose each sequence element $F^{\bf a}_k$ given by formula \eqref{30} into a part $D^{\bf a}_k$ containing linear combinations of the basic solutions \eqref{29} with the dominant moduli equal to $|k^{j_0}R^k|$, and a part $E^{\bf a}_k$ containing linear combinations of the basic solutions with moduli smaller than $|k^{j_0}R^k|$. Thus, for any $k\ge -n+1$, we have $F^{\bf a}_k=D^{\bf a}_k+E^{\bf a}_k$, where
\begin{equation} 
D^{\bf a}_k=k^{j_{0}}\sum_{l=1}^\eta c^{\bf a}_{lj_{0}}\lambda^k_l. \label{37}
\end{equation}

According to Theorem \ref{A}, if limit \eqref{33} exists, it is equal to a zero of the characteristic polynomial $P\!$, say $\tilde\lambda$. So, we obtain that 
\begin{equation}
\tilde\lambda=\lim_{k_0<k\to\infty}\frac{F^{\bf a}_{k+1}}{F^{\bf a}_k}=\lim_{k_0<k\to\infty}\frac{D^{\bf a}_{k+1}+E^{\bf a}_{k+1}}{D^{\bf a}_k+E^{\bf a}_k}=\lim_{k_0<k\to\infty}\frac{D^{\bf a}_{k+1}}{D^{\bf a}_k+E^{\bf a}_k}.\label{38}
\end{equation}
Formula \eqref{38} implies that 
\begin{equation}
\tilde\lambda^{-1}=\lim_{k_0<k\to\infty}\frac{D^{\bf a}_k+E^{\bf a}_k}{D^{\bf a}_{k+1}}=\lim_{k_0<k\to\infty}\frac{D^{\bf a}_k}{D^{\bf a}_{k+1}}.\label{39}
\end{equation}
It follows from \eqref{37} and \eqref{39} that 
\begin{equation} 
\lim_{k_0<k\to\infty}\biggr(\frac{\sum_{l=1}^\eta c^{\bf a}_{lj_0}(\lambda_l/R)^{k+1}}{\sum_{l=1}^\eta c^{\bf a}_{lj_0}(\lambda_l/R)^k}\biggr)= \tilde\lambda/R.\label{40}
\end{equation}

To simplify the notation, let us set $c^{\bf a}_l=c^{\bf a}_{lj_0}$, and let us introduce normalized zeros $\gamma_l=\lambda_l/R$, i.e., $|\gamma_l|=1$, $l=1,\dots,\eta$. 
Since limit \eqref{40} exists, the sequence 

\begin{equation} 
\biggr(\frac{\sum_{l=1}^\eta c^{\bf a}_l\gamma_l^{k+1}}{\sum_{l=1}^\eta c^{\bf a}_l\gamma_l^{k}}\biggr)^\infty_{k=k_0+1}\label{41}
\end{equation}
is a Cauchy sequence. Thus, for any $\epsilon>0$, there exist $k_{\epsilon}$ such that for $k>k_{\epsilon}$

\begin{equation} 
\biggr|\frac{\sum_{l=1}^\eta c^{\bf a}_l\gamma_l^{k+2}}{\sum_{l=1}^\eta c^{\bf a}_l\gamma_l^{k+1}}-\frac{\sum_{l=1}^\eta c^{\bf a}_l\gamma_l^{k+1}}{\sum_{l=1}^\eta c^{\bf a}_l\gamma_l^{k}}\biggr|<\epsilon. \label{42}
\end{equation} 
We transform inequality \eqref{42} into

\begin{equation} 
\biggr|\frac{\sum_{l=1}^\eta c^{\bf a}_l\gamma_l^{k+2}\sum_{l=1}^\eta c^{\bf a}_l\gamma_l^k-\big(\sum_{l=1}^\eta c^{\bf a}_l\gamma_l^{k+1}\big)^2}{\sum_{l=1}^\eta c^{\bf a}_l\gamma_l^{k+1}\sum_{l=1}^\eta c^{\bf a}_l\gamma_l^k}\biggr|=\label{43}
\end{equation}

\begin{equation}
=\biggr|\frac{\sum_{l=1}^\eta\sum_{m=l+1}^\eta c^{\bf a}_lc^{\bf a}_m\gamma^k_l\gamma^k_m(\gamma_l-\gamma_m)^2}{\sum_{l=1}^\eta c^{\bf a}_l\gamma_l^{k+1}\sum_{l=1}^\eta c^{\bf a}_l\gamma_l^k}\biggr|<\epsilon.\label{44}
\end{equation}
It follows from inequality \eqref{44} that the sequence 

\begin{equation} 
\biggr(\frac{\sum_{l=1}^\eta\sum_{m=l+1}^\eta c^{\bf a}_lc^{\bf a}_m\gamma^k_l\gamma^k_m(\gamma_l-\gamma_m)^2}{\sum_{l=1}^\eta c^{\bf a}_l\gamma_l^{k+1}\sum_{l=1}^\eta c^{\bf a}_l\gamma_l^k}\biggr)^\infty_{k=k_0+1}\label{45}
\end{equation}
must converge to 0. 

\smallskip The denominators in sequence \eqref{45} are bounded from above due to the fact that the moduli $|\gamma_l|=1$, $l=1,\dots,\eta$. Thus, the numerators of this sequence must form a sequence converging to 0.  
However, the sequences $(\gamma^k_l\gamma^k_m)_{k=k_0+1}^\infty$ oscillate for each pair of indices $(l,m)$, and therefore their linear combination does not converge to 0. Consequently, sequence \eqref{45} converges to 0 only when it is the constant sequence $(0)^\infty_{k=k_0+1}$, i.e., all normalized zeros $\gamma_l$ are equal one to another. 

So, if the ratio limit \eqref{33} exists, there can be only one zero that satisfies conditions (i) and (ii) listed above. The latter  contradicts our assumption that the characteristic polynomial $P^{\bf a}$ of the solution $F^{\bf a}$ is not asymptotically simple.

\end{proof}

\noindent MSC2010: 40A05, 40A30


\begin{thebibliography}{99}

\bibitem{Azarian} M.~K.~Azarian, Identities involving Lucas or
Fibonacci and Lucas numbers as binomial sums,
\emph{Int.~J.~Contemp.~Math.~Sci}.~\textbf{7} (2012), 2221--2227.

\bibitem{Balof} B.~Balof, Restricted tilings and bijections,
\emph{J.~Integer Seq.}~\textbf{15} (2012), 
\href{https://cs.uwaterloo.ca/journals/JIS/VOL15/Balof/balof19.html}{Article 15.4.5}.

\bibitem{Baltic} V.~Baltic, On the number of certain types of strongly restricted permutations, \emph{Appl.~Anal.~Discrete Math.}~\textbf{4} (2010), 119--135. 

\bibitem{Cahill} N.~D.~Cahill, J.~R.~D'Errico, and J.~P.~Spenc, Complex
factorizations of the Fibonacci and Lucas numbers, \emph {Fibonacci
Quart.}~\textbf{41} (2003), 13--19.

\bibitem{Chaffin} B.~Chaffin, J.~P.~Linderman, N.~J.~A.~Sloane, and
A.~R.~Wilks, On curling numbers of integer sequences, \emph{J.~Integer
Seq.}~\textbf{16} (2013), 
\href{https://cs.uwaterloo.ca/journals/JIS/VOL16/Sloane/sloane3.html}{Article 13.4.3}.

\bibitem{Dubeau1} F.~Dubeau, On $r$-generalized Fibonacci numbers, \emph {Fibonacci Quart.}~\textbf{27} (1989), 221--228. 

\bibitem{Dubeau4} F.~Dubeau, The rabbit problem revisited, \emph
{Fibonacci Quart.}~\textbf{31} (1993), 268--274.

\bibitem{Dubeau3} F.~Dubeau and A.~G.~Shannon, A Fibonacci model of
infectious disease, \emph {Fibonacci Quart.}~\textbf{34} (1996),
257--270.

\bibitem{Dubeau} F.\ Dubeau, W.\ Motta, M.\ Rachidi, and O.\ Saeki, 
 On weighted $r$-generalized Fibonacci sequences, \emph{Fibonacci Quart.}\ \textbf{35} (1997), 102--110.

\bibitem{Falcon} S.~Falcon, On the Lucas triangle and its relationship
with the $k$-Lucas numbers, \emph{J.~Math.~Comp.~Sci.}~\textbf{2} (2012),
425--434.

\bibitem{Guo} L.~Guo and W.~Sit, Enumeration and generating functions
for differential Rota-Baxter words,
\emph{Math.~Comput.~Sci.}~\textbf{4} (2010), 313--337.

\bibitem{Huang} J.~Huang, Hecke algebras with independent parameters,
published electronically at \url{http://arxiv.org/abs/1405.1636}.

\bibitem{Janjic} M.~Janjic, Determinants and recurrence sequences,
\emph{J.~Integer Seq.}~\textbf{15} (2012), 
\href{https://cs.uwaterloo.ca/journals/JIS/VOL15/Janjic/janjic42.html}{Article 12.3.5}.


\bibitem{Jeske} J.\ A.~Jeske. Linear recurrence relations, part I, \emph{Fibonacci Quart.}\ \textbf{1(2)} (1963), 69--74.

\bibitem{Kelly} W.\ G.\ Kelly and A.\ C.\ Peterson, \emph{Difference Equations: An Introduction with Applications},
 Acad.~Press, 1991.

\bibitem{Khovanova} T.\ Khovanova, \emph{Recursive sequences}, \\
\url{http://www.tanyakhovanova.com/RecursiveSequences/RecursiveSequences.html}

\bibitem{Linton} S.~Linton, J.~Propp, T.~Roby, and J.~West, Equivalence
classes of permutations under various relations generated by
constrained transpositions, \emph{J.~Integer Seq.}~\textbf{15} (2012),
\href{https://cs.uwaterloo.ca/journals/JIS/VOL15/Roby/roby4.html}{Article 12.9.1}.

\bibitem{Noe} T.~D.~Noe and J.~V.~Post, Primes in Fibonacci $n$-step
and Lucas $n$-step sequences, \emph{J.~Integer Seq.}~\textbf{8} (2005),
\href{https://cs.uwaterloo.ca/journals/JIS/VOL8/Noe/noe5.html}{Article 05.4.4}.

\bibitem{Maurin} K.\ Maurin, \emph{Analysis, Part I}, D.\ Reidel-PWN, 1976.

\bibitem{Rajan} A.~Rajan, R.~V.~Rao, A.~Rao, and H.~S.~Jamadagni,
Fibonacci sequence, recurrence relations, discrete probability
distributions and linear convolution, published electronically at 
\url{http://arxiv.org/abs/1205.5398}.

\bibitem{Shattuck} M.~A.~Shattuck and C.~G.~Wagner, Periodicity and
parity theorems for a statistic on $r$-mino arrangements,
\emph{J.~Integer Seq.}~\textbf{9} (2006), 
\href{https://cs.uwaterloo.ca/journals/JIS/VOL9/Shattuck/shattuck56.html}{Article 06.3.6}.

\bibitem{Sloane} N.\ J.\ A. Sloane, Online Encyclopedia of Integer Sequences, \url{http://oeis.org}

\bibitem{Szczyrba1} I.\ Szczyrba, R.\ Szczyrba, and M.\ Burtscher, Analytic representations  of the $\it n$-anacci 
constants and generalizations thereof, \emph{J.\ Integer Seq.}\ \textbf{18} (2015), Article 15.4.5.

\bibitem{Szczyrba2} I.\ Szczyrba, R.\ Szczyrba, and M.\ Burtscher, Geometric representations of the $\it n$-anacci 
constants and generalizations thereof, \emph{J.\ Integer Seq.}\ \textbf{19} (2016), Article 16.3.8.

\bibitem{Szczyrba3} I.\ Szczyrba, Asymptotic behavior of integer sequences related to knots and  
$(m,n)$-anacci constants, to appear in \emph{J.\ Knot Theory Ramifications}.

\end{thebibliography}
\end{document}